\documentclass[11pt,twoside,a4paper]{article}

\usepackage{amssymb}
\usepackage{amsmath}
\usepackage{amsthm}

\usepackage{mathrsfs}

\usepackage{amsfonts}
\usepackage{amssymb}
\usepackage{amsmath}
\usepackage{amsthm}
\usepackage{bbm}
\usepackage{verbatim}
\usepackage{url}

\allowdisplaybreaks

\def\C{{\mathbb C}}

\def\fz{\infty}

\def\ls{\lesssim}
\def\gs{\gtrsim}

\def\bmox{{{\rm BMO}_\nu(\C)}}

\newtheorem{thm}{Theorem}[section]
\newtheorem{defn}[thm]{Definition}

\newcommand{\R}{\mathbb R}

\numberwithin{equation}{section}

\begin{document}

\arraycolsep=1pt

\title{\Large\bf  Two weight commutators
 for Beurling--Ahlfors operator}
\author{Xuan Thinh Duong,  Ji Li, and Brett D. Wick}

\date{}
\maketitle

\begin{center}
\begin{minipage}{13.5cm}\small

{\noindent  {\bf Abstract:}\ We establish the equivalent characterisation of the weighted BMO space on the complex plane $\C$ via the 
two weight commutator of the Beurling--Ahlfors operator with a BMO function.  Our method of proofs relies on the explicit kernel of the  Beurling--Ahlfors operator and the properties of Muckenhoupt weight class.
}

\end{minipage}
\end{center}

%
%
\bigskip


\medskip


\section{Introduction and statement of main result\label{s1}}

The theory of singular integrals in harmonic analysis has had its origin closely related to other fields of mathematics such as 
complex analysis and partial differential equations. A typical example is the Hilbert transform which has arisen from the
complex conjugates of harmonic functions in the real and complex parts of analytic functions. The higher dimensional version of
the Hilbert transform is the Riesz transform on the Euclidean space $\mathbb R^n$. The Hardy space $H^1$ and its dual space,
the BMO space (BMO is abbreviation for bounded mean oscillation) have played an important role for the end-point estimates
of singular integrals as they are known as the substitutes of the spaces $L^1$ and $L^{\infty}$.

For singular integrals, weighted estimates are  important and the $A_p$  class of Muckenhoupt weights has provided the appropriate class of weights for the study of Calder\'on-Zygmund operators. One can  use two weight estimates on commutators of BMO functions with certain 
singular integral operators to characterise BMO spaces. The recent paper
\cite{HLW} gives a notable result which characterised weighted BMO spaces through two weight estimates on commutators of BMO functions 
and the Riesz transform.
More specifically, consider the $j$-th Riesz transform on  $\mathbb R^n$ given by $R_j = \frac{\partial}{\partial x_j} \Delta^{-1/2}$,
the weights $\lambda_{1}, \lambda_{2}$ in the Muckenhoupt class $A_p$, $1 < p < \infty$,  and the weight $\nu = \lambda_{1}^{1/p} \lambda_{2}^{-1/p}$. Denote by  $L^p_{w}(\R^n)$ the weighted $L^p$ space with the measure $w(x)dx$
and $[b, R_i] (f)(x) = b(x) R_i (f)(x) - R_i (bf)(x)$  the commutator of the Riesz transform $R_i$ and the function $b \in {\rm BMO}_\nu(\R^n)$, i.e., the Muckenhoupt--Wheeden weighted BMO space (introduced in \cite{MW76}).
The main result in \cite{HLW}, Theorem 1.2,  says that
there exist constants $0 < c < C < \infty$ such that
\begin{equation}\label{HLW}
c\| b \| _{{\rm BMO}_\nu(\R^n)} \le \sum _{i=1}^{n} \| [b, R_i] : L^p_{\lambda_{1}}(\R^n) \rightarrow L^p_{\lambda_{2}}(\R^n) \|  \le C \| b \| _{{\rm BMO}_\nu(\R^n)}
\end{equation}
where the constants $c$ and $C$  depend only on $n, p, \lambda_{1}, \lambda_{2}$. 
This result extended previous results of Bloom \cite{B}, Coifman, Rochberg and Weiss \cite{crw} and Nehari \cite{Ne}.


While the upper bound of the two weight commutator can be obtained for a large class of singular integral operators, the lower bound
is delicate and its proof for each specific operator can be quite  different and depends on the nature of the operator. For example, 
the proof for the lower bound of the commutator with the Riesz transform used the spherical harmonic expansions for the Riesz kernels, 
which relies on the property of the Fourier transform of the Riesz  kernels. 

In this paper, we consider the Beurling-Ahlfors operator 
$\mathcal B$ (see for example \cite{PV,VN}) which plays a notable role in complex analysis
and  is given by convolution with the distributional kernel p.v. $\displaystyle1\over\displaystyle z^2$, i.e.,  for $x\in \C$,
\begin{align*}
\mathcal{B}(f)(x)&= {\rm p.v.} {1\over \pi} \int_{\C} {f(y) \over \big( x-y\big)^2} \ dy.
\end{align*}
Here, for simplicity, we just use $dy$ to denote Lebesgue measure on $\C$. For other
works on the Beurling--Ahlfors operator, 
see for example \cite{PV} where they established a sharp weighted estimate of $\mathcal B$,
which is sufficient to prove that any weakly quasiregular map is quasiregular.

We now recall the Muckenhoupt--Wheeden type weighted BMO space on $\C$. For $\nu\in A_2(\C)$, ${\rm BMO}_{\nu}(\C)$ is defined (see \cite{MW76}) as the set of all $f\in L^1_{loc}(\C)$, such that $$ \|f\|_{{\rm BMO}_{\nu}(\C)}:=\sup_Q {1\over \nu(Q)}\int_Q\big| f- f_Q  \big|dx <\infty, $$
where the supremum is taken over all cubes $Q\subset \C$ and 
$$ f_Q := {1\over |Q|} \int_Q f(y)dy. $$

A natural question is  as follows.
\begin{itemize}
\item[Q:]
 Can we establish the characterisation of two weight commutator and the related weighted BMO space for the Beurling--Ahlfors operator, i.e. obtain (\ref{HLW}) with the Beurling--Ahlfors operator in place of the Riesz transform?
\end{itemize}

Our following main result gives a positive answer to this question.
 
\begin{thm}\label{thm3}
Suppose $1<p<\infty$, $\lambda_1,\lambda_2\in A_p(\C)$ and $\nu= \lambda_1^{1\over p}\lambda_2^{-{1\over p}}$. Suppose $b\in L^1_{\rm loc}(\C)$. Let $\mathcal B$ be the Beurling--Ahlfors operator.  Then we have
\begin{equation}\label{equiv}
 \| b \| _{{\rm BMO}_\nu(\C)}\approx \| [b, \mathcal B] : L^p_{\lambda_1}(\C) \rightarrow L^p_{\lambda_2}(\C) \|.
\end{equation}
\end{thm}

We provide the proof of the above theorem in Section 2. Then in Section 3, we provide the application of Theorem \ref{thm3}: the weak factorisation of 
the weighted Hardy space via the bilinear form in terms of the  Beurling--Ahlfors operator, which extends the classical result of Coifman, Rochberg and Weiss \cite{crw}.

\medskip
Throughout the paper,
we denote by $C$ and $\widetilde{C}$ {\it positive constants} which
are independent of the main parameters, but they may vary from line to
line. For every $p\in(1, \fz)$, we denote by $p'$ the conjugate of $p$, i.e., $\frac{1}{p'}+\frac{1}{p}=1$.  If $f\le Cg$, we then write $f\ls g$ or $g\gs f$;
and if $f \ls g\ls f$, we  write $f\approx g.$




\section{Proof of the main theorem}\label{s3}

We first recall the definition and some basic properties of the Muckenhoupt $A_p(\C)$ weights.

\begin{defn}
Suppose $w\in L^1_{loc}(\C)$, $w\geq0$, and $1<p<\infty$. We say that $w$ is a  Muckenhoupt $A_p(\C)$ weight
if there exists a constant $C$ such that
\begin{align} \label{Ap}
\sup_{Q} \left\langle w\right\rangle_Q   \left\langle w^{-{1\over p-1}}\right\rangle_Q^{p-1}\leq C<\infty,
\end{align}
where the supremum is taken over all cubes $Q$ in $\C$.
We denote by $[w]_{A_p}$ the smallest constant $C$ such that \eqref{Ap} holds.

The class $A_1(\C)$ consists of the weights $w$ satisfying for some $C > 0$ that
$$\left\langle w\right\rangle_Q \leq C {\rm ess}\inf_{x\in Q} w(x)$$
 for any cubes $Q\subset \C$.
We denote by $[w]_{A_1}$ the smallest constant $C$ such that the above inequality holds.
\end{defn}

If $w\in A_p(\C)$ with $p>1$, then the ``conjugate'' weight
\begin{align}\label{conjugate weight}
w'= w^{1-p'} \in A^{p'}(\C)
\end{align}
 with $[w']_{A^{p'}}= [w]_{A_p}^{p'-1}$, where $p'$ is the conjugate index of $p$, i.e., $1/p +1/p' =1$.
Moreover, suppose $\lambda_1,\lambda_2\in A_p(\C)$ with $1<p<\infty$.  Set
\begin{align}\label{Bloom weight}
\nu= \lambda_1^{1\over p}\lambda_2^{-{1\over p}}.
\end{align}
Then we have that $\nu\in A_2(\C)$, see \cite[Lemma 2.19]{HLW}. Moreover, we have the following
fundamental result (see \cite[equation (2.21)]{HLW}): for any ball $B \subset \C$,
\begin{align}\label{Bloom weight2}
\Big({\lambda_1(B)\over |B|}\Big)^{1\over p} \Big({\lambda'_2(B)\over |B|}\Big)^{{1\over p'}}\ls
{1\over  \Big({\lambda'_1(B)\over |B|}\Big)^{1\over p'} \Big({\lambda_2(B)\over |B|}\Big)^{{1\over p}}}
\ls {1\over   {\nu^{-1}(B)\over |B|} } \ls {\nu(B)\over |B|}.
\end{align}

Suppose $1<p<\infty$, $\lambda_{1},\lambda_{2}\in A_p(\C)$ and $\nu= \lambda_{1}^{1\over p}\lambda_{2}^{-{1\over p}}$. Note that $\nu\in A_2(\C)$.
Since $\mathcal B$ is a Calder\'on--Zygmund operator, following the result in \cite{HLW} we obtain that there exists a positive constant $C$ such that for
$b\in {\rm BMO}_\nu(\C)$, 
\begin{equation}\label{upper}
\| [b, \mathcal B] : L^p_{\lambda_1}(\C) \rightarrow L^p_{\lambda_2}(\C) \|\leq C  \| b \| _{{\rm BMO}_\nu(\C)}.
\end{equation}
When $\lambda_1=\lambda_2$, then the upper bound with precise information about the constant $C$ as a function of the $A_p$ characteristic was obtained in \cite{Chu}.

We now prove the lower bound. Suppose that $b\in L^1_{\rm loc}(\C)$ and that $[b, \mathcal B]$ is bounded from
$ L^p_{\lambda_{1}}(\C) $ to $ L^p_{\lambda_{2}}(\C)$. It suffices to show that for every cube $Q\subset \C$, there exists a positive constant $C$ such that 
\begin{align*}
{1\over \nu(Q)}\int_Q |b(x) - b_Q| dx \leq C <\infty.
\end{align*}

To see this,
without lost of generality, we now consider an arbitrary  cube $Q\subset \C$ centered at the origin.  Then we have 
\begin{align*}
&\int_Q |b(x) - b_Q| dx\\
&=\int_\C [b(x) - b_Q] {\rm\, sgn}(b(x)-b_Q) \chi_Q(x)dx\\
&={1\over |Q|} \int_\C\int_\C [b(x) - b(y)] {\rm\, sgn}(b(x)-b_Q)\chi_Q(y)dy\, \chi_Q(x)dx\\
&={1\over |Q|} \int_\C\int_\C [b(x) - b(y)] {\rm\, sgn}(b(x)-b_Q) {  (x-y)^2\over (x-y)^2 } \chi_Q(y)dy\, \chi_Q(x)dx\\
&=: I_1+I_2+I_3.
\end{align*}
where 
\begin{align*}
I_1&:={1\over |Q|} \int_\C\int_\C [b(x) - b(y)] {\rm\, sgn}(b(x)-b_Q) {  x^2\over (x-y)^2 } \chi_Q(y)dy\, \chi_Q(x)dx\\
I_2&:={1\over |Q|} \int_\C\int_\C [b(x) - b(y)] {\rm\, sgn}(b(x)-b_Q) {  -2xy\over (x-y)^2 } \chi_Q(y)dy\, \chi_Q(x)dx\\
I_3&:={1\over |Q|} \int_\C\int_\C [b(x) - b(y)] {\rm\, sgn}(b(x)-b_Q) {  y^2\over (x-y)^2 } \chi_Q(y)dy\, \chi_Q(x)dx.
\end{align*}

We now denote  $\Gamma_Q(x):= {\rm\, sgn}(b(x)-b_Q)$. Then for the term $I_1$, we obtain that
\begin{align*}
|I_1|&={1\over |Q|} \bigg|\int_\C [b,\mathcal B] (\chi_Q)(x) \, x^2  \Gamma_Q(x)\chi_Q(x)dx\bigg|\\
&\leq {1\over |Q|} \Big\| [b,\mathcal B] (\chi_Q)\Big\|_{L^p_{\lambda_2}(\C)}   \Big\| x^2  \Gamma_Q(x)\chi_Q(x) \Big\|_{L^{p'}_{\lambda'_2}(\C)}\\
&\leq C\| [b, \mathcal B] : L^p_{\lambda_1}(\C) \rightarrow L^p_{\lambda_2}(\C) \| \Big\| \chi_Q\Big\|_{L^p_{\lambda_1}(\C)}\Big\| \chi_Q \Big\|_{L^{p'}_{\lambda'_2}(\C)}\\
&\leq C\| [b, \mathcal B] : L^p_{\lambda_1}(\C) \rightarrow L^p_{\lambda_2}(\C) \| \lambda_1(Q)^{1\over p} \lambda'_2(Q)^{1\over p'}\\
&\leq C\| [b, \mathcal B] : L^p_{\lambda_1}(\C) \rightarrow L^p_{\lambda_2}(\C) \|\nu(Q),
\end{align*}
where in the first inequality we use Holder's inequality with the index ${1\over p}+{1\over p'}=1$, in the second inequality we use the boundedness of $[b, \mathcal B]$  from
$ L^p_{\lambda_{1}}(\C) $ to $ L^p_{\lambda_{2}}(\C)$ and the fact that $|\Gamma_Q(x)|\leq 1$ for any $x\in \C$, and in the last inequality we use the fundamental fact in 
\eqref{Bloom weight2}.

As for the term $I_2$, similarly, we have
\begin{align*}
|I_2|&:={2\over |Q|} \bigg| \int_\C\int_\C [b,\mathcal B] \big( y\chi_Q (y)\big) (x)\ x \Gamma_Q(x) \chi_Q(x)dx \bigg|\\
&\leq {1\over |Q|} \Big\| [b,\mathcal B] \big(y\chi_Q(y)\big)\Big\|_{L^p_{\lambda_2}(\C)}   \Big\| x  \Gamma_Q(x)\chi_Q(x) \Big\|_{L^{p'}_{\lambda'_2}(\C)}\\
&\leq C\| [b, \mathcal B] : L^p_{\lambda_1}(\C) \rightarrow L^p_{\lambda_2}(\C) \|{1\over |Q|} \Big\| y \chi_Q(y)\Big\|_{L^p_{\lambda_1}(\C)}\Big\| x \chi_Q(x) \Big\|_{L^{p'}_{\lambda'_2}(\C)}\\
&\leq C\| [b, \mathcal B] : L^p_{\lambda_1}(\C) \rightarrow L^p_{\lambda_2}(\C) \| \lambda_1(Q)^{1\over p} \lambda'_2(Q)^{1\over p'}\\
&\leq C\| [b, \mathcal B] : L^p_{\lambda_1}(\C) \rightarrow L^p_{\lambda_2}(\C) \|\nu(Q).
\end{align*}
Again, for the term $I_3$, using similar argument, we get that
\begin{align*}
|I_3|&:={1\over |Q|} \bigg| \int_\C\int_\C [b,\mathcal B] \big( y^2\chi_Q (y)\big) (x)\  \Gamma_Q(x) \chi_Q(x)dx \bigg|\\
&\leq {1\over |Q|} \Big\| [b,\mathcal B] \big(y^2\chi_Q(y)\big)\Big\|_{L^p_{\lambda_2}(\C)}   \Big\|   \Gamma_Q(x)\chi_Q(x) \Big\|_{L^{p'}_{\lambda'_2}(\C)}\\
&\leq C\| [b, \mathcal B] : L^p_{\lambda_1}(\C) \rightarrow L^p_{\lambda_2}(\C) \|{1\over |Q|} \Big\| y^2 \chi_Q(y)\Big\|_{L^p_{\lambda_1}(\C)}\Big\|  \chi_Q(x) \Big\|_{L^{p'}_{\lambda'_2}(\C)}\\
&\leq C\| [b, \mathcal B] : L^p_{\lambda_1}(\C) \rightarrow L^p_{\lambda_2}(\C) \| \lambda_1(Q)^{1\over p} \lambda'_2(Q)^{1\over p'}\\
&\leq C\| [b, \mathcal B] : L^p_{\lambda_1}(\C) \rightarrow L^p_{\lambda_2}(\C) \|\nu(Q).
\end{align*}

As a consequence, combining the estimates of $I_1$, $I_2$ and $I_3$, we get that
\begin{align*}
\int_Q |b(x) - b_Q| dx\leq C\| [b, \mathcal B] : L^p_{\lambda_1}(\C) \rightarrow L^p_{\lambda_2}(\C) \| \nu(Q),
\end{align*}
which implies that
$$ \|b\|_{{\rm BMO}_\nu(\C)} \leq C\| [b, \mathcal B] : L^p_{\lambda_1}(\C) \rightarrow L^p_{\lambda_2}(\C) \|. $$

\section{Applications: Weak factorization of the weighted Hardy space}\label{s3.1}

We recall the weighted Hardy space, then prove that it is the predual of $\bmox$.
Note that based on the results in \cite{MW76} and the recent result in \cite{DHLWY}, 
there are also other equivalent characterisations of the weighted Hardy space, for example, via Littlewood--Paley area functions, maximal functions, etcetera.  For simplicity, we just define the weighted Hardy space via atoms as follows.

\begin{defn}\label{d-hardy -2}
Suppose $\nu \in A_2(\C)$.
A function $a\in L^2(\C)$ is called an $\nu$-weighted $(1,2)$-atom if it satisfies
\begin{enumerate}
\item[(1)] {\rm supp}  $a\subset B$, where $B$ is a ball in $\C$;
\item[(2)] $\int_\C a(x)dx=0$;
\item[(3)] $ \|a\|_{L^2_\nu(\C)}\leq \nu(B)^{-{1\over 2}} $.
\end{enumerate}
We say that $f$
belongs to
the weighted Hardy space $H^{1}_\nu(\C)$ if
$f$ can be written as
\begin{equation}\label{d atom 2}
f=\sum_j \alpha_j a_j
\end{equation}
with $\sum_j|\alpha_j|<\infty$. The $H^{1}_\nu(\C)$ norm of $f$ is defined as
$$ \|f\|_{H^{1}_\nu(X)}:=\inf\Big\{\sum_j|\alpha_j|:\, f {\rm\ \ has\ the \ representation\ as\ in\ } \eqref{d atom 2}  \Big\}. $$
\end{defn}

We also recall the John--Nirenberg inequality for the ${\rm BMO}_\nu(\C)$. According to \cite{DHLWY}, we know that for $v\in A_2(\C)$ and for $1\leq r\leq 2$,
\begin{align}
\|b\|_{{\rm BMO}_{\nu}(\C)} \leq \|b\|_{{\rm BMO}_{\nu,r}(\C)} \leq C_{n,p,r} [\nu]_{A_2}\|b\|_{{\rm BMO}_{\nu}(\C)},
\end{align}
where
\begin{align}
\|b\|_{{\rm BMO}_{\nu,r}(\C)}:=
\bigg(\sup_Q\frac1{\nu(Q)}\int_Q\left|b(x)-b_Q\right|^r\, \nu^{1-r}(x)dx\bigg)^{1\over r}.
\end{align}

\begin{thm}\label{thm dual}
Suppose $\nu \in A_2(\C)$. The dual of $H^{1}_\nu(\C)$ is ${\rm BMO}_\nu(\C)$.
\end{thm}
\begin{proof}
This duality result follows from a standard argument, see for example \cite{cw77}.  
By completeness, we provide the proof as follows.
We first show that  $${\rm BMO}_{\nu}(\C)\subset \big(H^{1}_{\nu}(\C)\big)^\ast.$$
In fact, for any $g\in {\rm BMO}_{\nu}(\C)$,
define
$$L_g(a):=\int_{\C} a(x)g(x)dx,$$
where  $a$ is an $\nu$-weighted $(1,2)$-atom.

Assume that $a$ is supported in a cube $Q$.
Then from H\"older's inequality and $\nu\in A_2(\C)$, we see that
\begin{align*}
\left|\int_{\C} a(x)g(x)\,dx\right|
&=\left|\int_Q a(x)[g(x)-g_Q]\,dx\right|\\
&\le \left(\int_Q |g(x)-g_Q|^2\nu^{-1}(x)\,dx\right)^{\frac12}\left(\int_Q |a(x)|^2\nu(x)\,dx\right)^{\frac12}\\
&\le \left[\frac1{\nu(B)}\int_Q |g(x)-g_Q|^2\nu^{-1}(x)\,dx\right]^{\frac12}\\
&\le C\|g\|_{{\rm BMO}_{\nu}(\C)}.
\end{align*}
Thus $L_g$ extends to a bounded linear functional on $H^{1}_{\nu}(\C)$.

Conversely, assume that $L\in \big(H^{1}_{\nu}(\C)\big)^\ast$.
For any cube $Q$, let
$$L^2_{0,\,\nu}(Q)=\{f\in L^2_{\nu}(Q): {\rm supp}(f)\subset Q,\,\,\int_Q f(x)\,dx=0\}.$$
Then we see that for any $f\in L^2_{0,\,\nu}(Q)$, the function $\frac1{\nu(Q)^\frac12\|f\|_{L^2_{\nu}(Q)}}f$
is an $H^{1}_{\nu}(\C)$-atom.
This implies that
$$|L(a)|\le \|L\|\|a\|_{H^{1}_{\nu}(\C)}\le \|L\|.$$
Moreover, we see that
$$|L(f)|\le \|\ell\|\nu(Q)^\frac12\|f\|_{L^2_{\nu}(Q)}.$$
From the Riesz representation theorem, there exists $[\varphi]\in [L^2_{0,\,\nu}(Q)]^\ast=L^2_{0,\,\nu^{-1}}(Q)/\mathbb C$,
and $\varphi\in [\varphi]$, such that for any $f\in L^2_{0,\,\nu}(Q)$,
$$L(f)=\int_Q f(x)\varphi(x)dx$$
and
$$\|[\varphi]\|=\inf_c\|\phi+c\|_{L^2_{\nu^{-1}}(Q)}\le \|L\|\nu(Q)^\frac12.$$

Let $Q$ fixed and $Q_j=2^jQ$, $j\in\mathbb N$. Then we have that for all $f\in L^2_{0,\,\nu}(Q)$
and $j\in\mathbb N$,
$$\int_Q f(x)\varphi_{Q}(x)\,dx=\int_Q f(x)\varphi_{Q_j}(x)\,dx.$$
It follows that for almost every $x\in Q$, $\varphi_{Q_j}(x)-\varphi_{Q_0}(x)=C_j$
for some constant $C_j$. From this we further deduce that for all $j,\,l\in\mathbb N$,
$j\le l$ and almost every $x\in Q_j$,
$$\varphi_{Q_j}(x)-C_j=\varphi_{Q_0}(x)=\varphi_{Q_l}(x)-C_l.$$
Define
$$\varphi(x)=\varphi_j(x)-C_j$$
on $B_j$ for $j\in\mathbb N$. Thus, $\varphi$ is well defined. Moreover,
since $\C=\cup_j Q_j$,
by H\"older's inequality and $\nu\in A_2(\C)$,
we see that for any $c$ and any cube $Q$,
\begin{align*}
&\left[\int_Q|\varphi(x)-\varphi_Q|^2\nu^{-1}(x)\,dx\right]^{\frac12}\\
&\quad=
\sup_{\|f\|_{L^2_{\nu}(Q)}\le1}\left|\langle f, \varphi-\varphi_Q\rangle\right|\\
&\quad=
\sup_{\|f\|_{L^2_{\nu}(Q)}\le1}\left|\int_Q f(x)[\varphi(x)-\varphi_Q]\,dx\right|\\
&\quad=\sup_{\|f\|_{L^2_{\nu}(Q)}\le1}\left|\int_Q [f(x)-f_Q][\varphi(x)+c]\,dx\right|\\
&\quad\le\sup_{\|f\|_{L^2_{\nu}(Q)}\le1}\left[\|f\|_{L^2_{\nu}(Q)}
+|f_Q|\nu(Q)^\frac12\right]\|[\varphi(x)+c]\chi_Q\|_{L^2_{\nu^{-1}}(Q)}\\
&\quad\le \|[\varphi(x)+c]\chi_Q\|_{L^2_{\nu^{-1}}(Q)}.
\end{align*}
Taking the infimum over $c$, we have that
$\varphi\in {\rm BMO}_{\nu}(\C)$ and
$\|\varphi\|_{{\rm BMO}_{\nu}(\C)}\le C \|L\|$.
\end{proof}

%
%


The main result of this section is as follows.
\begin{thm}\label{weakfactorization}
Suppose $1<p<\infty$, $\lambda_1,\lambda_2\in A_p(\C)$ and $\nu= \lambda_1^{1\over p}\lambda_2^{-{1\over p}}$. For every $f\in H^{1}_\nu(\C)$,
there exist sequences $\{\alpha_j^k\}_j \in\ell^1$ and functions $h_j^k\in L^p_{\lambda_1}(\C)$, $g_j^k \in L^{p'}_{\lambda'_2}(\C)$ with $p'={p\over p-1}$ and $\lambda'_2=\lambda_2^{-{1\over p-1}}$ such that
\begin{align}
\label{repf}
f(x)= \sum_{k=1}^\infty  \sum_{j=1}^\infty \alpha_j^k \Pi(g_j^k,h_j^k)(x)
\end{align}
in the sense of $H^{1}_\nu(\C)$,
where $\Pi(g_j^k,h_j^k)(x)$ is the bilinear form defined as
$$\Pi(g_j^k,h_j^k)(x) :=h_j^k(x) \mathcal B (g_j^k)(x)- g_j^k(x) \mathcal B ^*(h_j^k)(x) .
$$
Moreover, we have that
$$\inf\left\{ \sum_{k=1}^\infty\sum_{j=1}^\infty  |\alpha_j^k | \|g_j^i\|_{L^{p'}_{\lambda'_2}(\C)}\|h_j^i\|_{L^p_{\lambda_1}(\C)}\right\}
\approx  \|f\|_{H^{1}_\nu(\C)},$$
where the infimum is taken over all possible representations of $f$ from \eqref{repf}.
\end{thm}
It is well known that this theorem follows from the duality between $H^1_{\nu}(\C)$ and ${\rm BMO}_{\nu}(\C)$ and the equivalence between ${\rm BMO}_{\nu}(\C)$ and the boundedness of the commutator, provided in Theorem \ref{thm3}.  We omit the details of this proof.

\bigskip
\bigskip

{\bf Acknowledgments:}
X. T. Duong and J. Li are supported by Australian Research Council DP 160100153.  B. D. Wick's research supported in part by National Science Foundation DMS grant \#1560955.

\medskip


\smallskip

Xuan Thinh Duong, Department of Mathematics, Macquarie University, NSW, 2109, Australia.

\smallskip

{\it E-mail}: \texttt{xuan.duong@mq.edu.au}

\vspace{0.3cm}

%
%
%
%
%



Ji Li, Department of Mathematics, Macquarie University, NSW, 2109, Australia.

\smallskip

{\it E-mail}: \texttt{ji.li@mq.edu.au}

\vspace{0.3cm}

%
%
%
%
%



Brett D. Wick, Department of Mathematics, Washington University--St. Louis, St. Louis, MO 63130-4899 USA

\smallskip

{\it E-mail}: \texttt{wick@math.wustl.edu}

\vspace{0.3cm}

\end{document}